\documentclass[reqno]{amsart}
\usepackage{amsmath, amsthm, amssymb, mathabx}
\begin{document}
\newtheorem{theo}{Theorem}[section]
\newtheorem{defin}[theo]{Definition}
\newtheorem{rem}[theo]{Remark}
\newtheorem{lem}[theo]{Lemma}
\newtheorem{cor}[theo]{Corollary}
\newtheorem{prop}[theo]{Proposition}
\newtheorem{exa}[theo]{Example}
\newtheorem{exas}[theo]{Examples}
%%%%%%%%%%%%%%%%%%%%%%%%%%%%%%%%%%%
%%%%%%%%%%%%%%%%%%%%%%%%%%%%%%%%%%%%
%
%
\subjclass[2010]{Primary 35J60  Secondary 35J25}
\keywords{Nonlocal problems; Kirchhoff's equation; Ground state solution; Nehari manifold}
\thanks{}
\title[Ground state solution of a nonlocal boundary value problem]{Ground state solution of a nonlocal boundary value problem}

\author[C. J. Batkam]{Cyril Joel Batkam}
\address{Cyril Joel Batkam \newline
D\'epartement de math\'ematiques,
\newline
Universit\'e de Sherbrooke,
\newline
Sherbrooke, Qu\'ebec, J1K 2R1, CANADA.}
\email{cyril.joel.batkam@usherbrooke.ca}

\maketitle
\begin{abstract}
In this paper, we apply the method of the Nehari manifold to study the Kirchhoff type equation
\begin{equation*}
-\Big(a+b\int_\Omega|\nabla u|^2dx\Big)\Delta u=f(x,u)
\end{equation*}
submitted to Dirichlet boundary conditions. Under a general $4-$superlinear condition on the nonlinearity $f$, we prove the existence of a ground state solution; that is a nontrivial solution which has least energy among the set of nontrivial solutions. In case which $f$ is odd with respect to the second variable, we also obtain the existence of infinitely many solutions. Under our assumptions the Nehari manifold does not need to be of class $\mathcal{C}^1$.
\end{abstract}

%
%%%%%%%%%%%%%%%%%%%%%%%%%%%%%%%%%%%%%%%%%%%%%%%%%%%%
%%%%%%%%%%%%%%%%%%%%%%%%%%%%%%%%%%%%%%%%%%%%%%%%%%%%
\section{Introduction}
%%%%%%%%%%%%%%%%%%%%%%%%%%%%%%%%%%%%%%%%%%%%%%%%%%%%
%%%%%%%%%%%%%%%%%%%%%%%%%%%%%%%%%%%%%%%%%%%%%%%%%%%%
In this paper, we will be concerned with the nonlocal boundary value problem
\begin{equation}\label{s}\tag{S}
    \left\{
      \begin{array}{ll}
        -\Big(a+b\int_\Omega|\nabla u|^2dx\Big)\Delta u=f(x,u),\quad \text{in }\Omega & \hbox{} \\
        u=0 \text{ on }\partial\Omega, & \hbox{}
      \end{array}
    \right.
\end{equation}
where $\Omega$ is a bounded domain in $\mathbb{R}^N$ with smooth boundary ($N=1,2,3$), $a>0$ and $b>0$, and the nonlinearity $f:\overline{\Omega}\times\mathbb{R}\to\mathbb{R}$ satisfies the following conditions.
\begin{enumerate}
  \item[$(f_1)$] $f$ is continuous and there exists a constant $c>0$ such that $|f(x,u)|\leq c\big(1+|u|^{p-1}\big)$, where $p>4$ for $N=1,2$ and $4<p<2^\star:=2N/(N-2)$ for $N=3$.
  \item[$(f_2)$] $f(x,u)=\circ (u)$ uniformly in $x$ as $|u|\to0$.
  \item[$(f_3)$] $F(x,u)/u^4\to\infty$ uniformly in $x$ as $|u|\to\infty$, where $F(x,u)=\int_0^u f(x,s)ds$.
  \item[$(f_4)$] $u\mapsto f(x,u)/u^3$ is positive for $u\neq0$, nonincreasing on $(-\infty,0)$ and nondecreasing on $(0,\infty)$.
\end{enumerate}
 Usually, in the study of \eqref{s} the following Ambrosetti-Rabinowitz's type condition is used:
\begin{equation}\label{AR}\tag{AR}
    \exists\mu>4,\,\, R>0\,\,;\,\, 0<\mu F(x,u)\leq uf(x,u)\quad\forall x\in\overline{\Omega},\,\,|u|\geq R.
\end{equation}
Integrating \eqref{AR} yields the existence of constants $c_1,c_2>0$ such that $F(x,u)\geq c_1|u|^\mu-c_2$ for all $u$, therefore \eqref{AR} is stronger that $(f_3)$. It is well known that \eqref{AR} is mainly used to verify the boundedness of the Palais-Smale sequences of the energy functional, and without it the problem becomes more complicated. However, there are many functions which are $4-$superlinear but do not satisfy \eqref{AR}. An example of $f$ satisfying assumptions $(f_1)-(f_4)$ which does not satisfy \eqref{AR} is given at the end of the paper.
\par We call problem \eqref{s} nonlocal because of the presence of the term $\int_\Omega|\nabla u|^2dx,$ which implies that the first equation in \eqref{s} is no longer a pointwise equality. This causes some mathematical difficulties which make the study of such problems particularly interesting. On the other hand, for a physical point of view problem \eqref{s} is related to the stationary analogue of the hyperbolic equations
\begin{equation*}
  u_{tt}-\Big(a+b\int_\Omega|\nabla u|^2dx\Big)\Delta u=f(x,u),
\end{equation*}
proposed by Kirchhoff \cite{Kir} as an extension of the classical d'Alembert wave equations for free vibrations of elastic strings. The Kirchhoff's model takes into account the changing in length of the string produced by transverse vibrations. Problem \eqref{s} has been widely studied by variational methods since the paper of Lions \cite{Lions}, where an abstract framework to attack it was introduced. In \cite{Per-Z} Perera and Zhang considered \eqref{s} in the case that $f$ is asymptotically linear at $0$ and asymptotically $4-$linear at infinity, and they obtained a nontrivial solution by using the Yang index and
critical group. In \cite{He-Zou}, under condition \eqref{AR} and without condition \eqref{AR}, He and Zou obtained the existence of infinitely many solutions of \eqref{s} by using the fountain theorems. In \cite{Alves} Alves et al. considered \eqref{s} with a critical term and obtained a nontrivial solution of mountain pass type. In \cite{Chen,WLJF,Junior} the authors obtained some existence results for a Kirchhoff's type problem by using the Nehari manifold approach.
\par In this paper, we also study \eqref{s} via a reduction on the Nehari manifold. We are firstly interested in the existence of a ground state solution of \eqref{s}; that is a nontrivial solution which has least energy among the set of nontrivial solutions of \eqref{s}. Let $X:=H_0^1(\Omega)$ be the usual Sobolev space endowed with the inner product $\big<\cdot,\cdot\big>$ and the associated norm $\|\cdot\|$,
\begin{equation*}
    \big<u,v\big>=\int\limits_\Omega\nabla u\nabla vdx,\quad\quad \|u\|^2=\big<u,u\big>.
\end{equation*}
Under assumption $(f_1)$, the solutions of \eqref{s} are critical points of the functional $\Phi\in\mathcal{C}^1(X,\mathbb{R})$,
\begin{equation}\label{phi}
   \Phi(u)=\frac{a}{2}\|u\|^2+\frac{b}{4}\|u\|^4-\int\limits_\Omega F(x,u)dx.
\end{equation}
 We define the Nehari manifold
\begin{equation}\label{nehari}
    \mathcal{N}:=\Big\{u\in X\backslash\{0\}\,\,;\,\, \big<\Phi'(u),u\big>=0\Big\}.
\end{equation}
Then any nontrivial solution of \eqref{s} belongs to $\mathcal{N}$. We would like to show that $\inf_{\substack{\mathcal{N}}}\Phi$ is attained at some $u_0\in\mathcal{N}$ which is a critical point of $\Phi$. Since under our assumptions on $f$ above we do not know if the sub-manifold $\mathcal{N}$ of $X$ is of class $\mathcal{C}^1$, we cannot apply the minimax theorems in \cite{R,S,W} directly to $\mathcal{N}$ in order to extract the critical points of the functional $\Phi$. To circumvent this difficulty we follow, as in \cite{WLJF,Junior}, an approach by Szulkin and Weth \cite{S-W1}.
Our main result is the following.
\begin{theo}\label{main}
Let $a>0$ and $b>0$. If $f$ satisfies $(f_1)-(f_4)$, then \eqref{s} has a ground state solution.
Moreover, if in addition
\begin{enumerate}
  \item [$(f_5)$] $f(x,-u)=-f(x,u)$ for all $(x,u)\in\overline{\Omega}\times\mathbb{R}$,
\end{enumerate}
then \eqref{s} has infinitely many solutions.
\end{theo}
%\par The paper is organized as follows. We present some preliminary materials in section \ref{section1},  and we prove Theorem \ref{main} in section \ref{section2}.

%%%%%%%%%%%%%%%%%%%%%%%%%%%%%%%%%%%%%%%%%%%%%%%%%%%%
%%%%%%%%%%%%%%%%%%%%%%%%%%%%%%%%%%%%%%%%%%%%%%%%%%%%
\section{Preliminaries}\label{section1}
%%%%%%%%%%%%%%%%%%%%%%%%%%%%%%%%%%%%%%%%%%%%%%%%%%%%
%%%%%%%%%%%%%%%%%%%%%%%%%%%%%%%%%%%%%%%%%%%%%%%%%%%%
Throughout the paper we denote by $|\cdot|_r$ the norm of the Lebesgue space $L^r(\Omega)$.\\
We consider $X$, $\Phi$ and $\mathcal{N}$ as defined in the introduction.\\
A standard argument shows that:
\begin{lem}
If $(f_1)$ is satisfied, then $\Phi\in\mathcal{C}^1(X,\mathbb{R})$ and we have
\begin{equation}\label{phiprime}
    \big<\Phi'(u),v\big>=\big(a+b\|u\|^2\big)\int\limits_\Omega\nabla u\nabla vdx-\int\limits_\Omega vf(x,u)dx.
\end{equation}
\end{lem}
\par It is well known that the Nehari manifold $\mathcal{N}$ is closely linked to the behavior of the map $\alpha_u:[0,\infty)\to\mathbb{R}$ defined by
\begin{equation}\label{alpha}
    \alpha_u(t):=\Phi(tu),
\end{equation}
where $u\in X$ is fixed. Such a map is known as a fibering map which was introduced by Dr\'{a}bek and Pohozaev in \cite{Dra} and discussed in Brown and Zhang \cite{Bro}. The following result shows that $\alpha_u$ has a unique maximum point if $u\neq0$.
\begin{lem}\label{tu}
Assume that $(f_1)-(f_4)$ are satisfied. Then for any $u\in X\backslash\{0\}$, there exists a unique $t_u>0$ such that $\alpha'_u(t)>0$ for every $t\in(0,t_u)$ and $\alpha'_u(t)<0$ for every $t>t_u$.
\end{lem}
\begin{proof}
$(f_1)$ and $(f_2)$ imply that
\begin{equation}\label{ine}
  \forall\varepsilon>0,\,\, \exists c_\varepsilon>0 \,\,;\,\ |f(x,u)|\leq \varepsilon|u|+c_\varepsilon|u|^{p-1} \text{ and } |F(x,u)|\leq \varepsilon|u|^2+c_\varepsilon|u|^p.
\end{equation}
We then deduce, using \eqref{phi} that
\begin{equation*}
    \alpha_u(t)\geq \Big(\frac{a}{2}\|u\|^2-\varepsilon |u|_2^2\Big)t^2+\frac{b}{4}\|u\|^4t^4-c_\varepsilon t^p|u|_p^p.
\end{equation*}
Since $u\neq0,$ we can choose $\varepsilon$ in such a way that $\frac{a}{2}\|u\|^2-\varepsilon |u|_2^2>0$. It then follows, since $p>4$, that $\alpha_u(t)>0$ for $t>0$ sufficiently small.\\
On the other hand $(f_3)$ implies:
\begin{equation}\label{ine1}
    \forall \delta>0,\,\, \exists c_\delta>0 \,\,;\,\, F(x,u)\geq\delta|u|^4-c_\delta.
\end{equation}
This implies that
\begin{equation*}
    \alpha_u(t)\leq \frac{a}{2}t^2\|u\|^2+\frac{b}{4}t^4\|u\|^4-\delta t^4|u|_4^4+c_\delta|\Omega|.
\end{equation*}
If we choose $\delta>0$ big enough such that $\frac{b}{4}\|u\|^4-\delta|u|_4^4<0$, we see that $\alpha_u(t)\to-\infty$ as $t\to\infty$.
We deduce that $\alpha_u$ has a positive maximum.
\par Now noting that
\begin{equation*}
    \alpha'_u(t)=\big<\Phi'(tu),u\big>=a\|u\|^2t+b\|u\|^4t^3-\int_\Omega uf(x,tu)dx,
\end{equation*}
the equation $\alpha'_u(t)=0$ is equivalent to
\begin{equation*}
    b\|u\|^4=-\frac{a\|u\|^2}{t^2}+\frac{1}{t^3}\int\limits_\Omega uf(x,tu)dx.
\end{equation*}
By $(f_4)$ the map $t\mapsto\frac{1}{t^3}\int\limits_\Omega uf(x,tu)dx$ is increasing on $(0,\infty)$. It is then easy to deduce that the map $t\mapsto-\frac{a\|u\|^2}{t^2}+\frac{1}{t^3}\int\limits_\Omega uf(x,tu)dx$ is strictly increasing on $(0,\infty)$. Hence the maximum point of $\alpha_u$ is unique.
\end{proof}
The following lemma gives some properties of $t_u.$  Let
\begin{equation*}
    S:=\big\{u\in X\,;\,\|u\|=1\big\}.
\end{equation*}
\begin{lem}\label{tusuite}
If $(f_1)-(f_4)$ are satisfied then:
\begin{enumerate}
  \item There exists $\delta>0$ such that $t_u\geq\delta$ for every $u\in S$, where $t_u$ is as in Lemma \ref{tu} above.
  \item For any compact $K\subset S$, there exists a constant $C_K$ such that $t_u\leq C_K$ for every $u\in K$.
\end{enumerate}
\end{lem}
\begin{proof}
\begin{enumerate}
  \item Let $u\in S$ and recall that $t_u$ is the unique point of maximum of the map $\alpha_u(t)=\Phi(tu)$. We deduce from \eqref{ine} and the Sobolev embedding theorem that
\begin{equation*}
    \Phi(w)\geq \Big(\frac{a}{2}-\varepsilon c_1\Big)\|w\|^2+\frac{b}{4}\|w\|^4-c_\varepsilon c_2\|w\|^p,\quad \forall w\in X\backslash\{0\},
\end{equation*}
where $c_1>0$ and $c_2>0$ are constants. By choosing $\varepsilon$ such that $\frac{a}{2}-\varepsilon c_1\geq\frac{a}{4}$, we obtain
\begin{equation*}
    \Phi(w)\geq \frac{a}{4}\|w\|^2+\frac{b}{4}\|w\|^4-c_3\|w\|^p,\quad \forall w\in X\backslash\{0\}.
\end{equation*}
There then exists  $\delta>0$ sufficiently small such that setting $w=\delta u$ we obtain
\begin{equation*}
    \Phi(\delta u)\geq\frac{a}{4}\delta^2+\frac{b}{4}\delta^4-c_3\delta^p>0,\quad\forall u\in S.
\end{equation*}
 Since $t_u>0$ is the unique point of maximum of the function $\alpha_u$ we have
\begin{equation*}
    \alpha_u(t_u)\geq \alpha_u(\delta)\geq \delta_\star:=\frac{a}{4}\delta^2+\frac{b}{4}\delta^4-c_3\delta^p>0,\quad\forall u\in S,
\end{equation*}
where $\delta_\star>0$ does not depend on $u\in S$. We would like to show that $t_u\geq\gamma>0$ for some $\gamma>0$ and for all $u\in S$. Suppose, on the contrary, that there is a sequence $(t_{u_j},u_j)$, with $u_j\in S$ such that $t_{u_j}\to0^+$. Since $u_j\in S$, we have $t_{u_j}u_j\to0$ in $X$ and so, in view of the continuity of $\Phi$, we obtain
\begin{equation*}
    0<\delta_\star\leq\alpha_{u_j}(t_{u_j})=\Phi(t_{u_j}u_j)\to0=\Phi(0)
\end{equation*}
which is a contradiction. Hence, there is $\gamma>0$ such that $t_u\geq\gamma>0$ for all $u\in S$.
  \item Let $K$ be a compact subset of $S$. Arguing by contradiction, we assume that there exists a sequence $(u_n)\subset K$ such that $t_{u_n}\to\infty.$
% This implies that $\|t_{u_n}u_n\|\to\infty$.
We know that there exists $\delta>0$ such that $\Phi(t_{u_n}u_n)\geq \Phi(\delta u_n)>0.$ Hence we have
\begin{align*}
    0<\frac{\Phi(t_{u_n}u_n)}{t^4_{u_n}}&=\frac{1}{t^4_{u_n}}\Big[\frac{a}{2}\|t_{u_n}u_n\|^2+\frac{b}{4}\|t_{u_n}u_n\|^4-\int\limits_\Omega F(x, t_{u_n}u_n)dx\Big]\\
&=\frac{a}{2t^2_{u_n}}+\frac{b}{4}-\int\limits_\Omega \frac{F(x,t_{u_n}u_n)}{t^4_{u_n}}dx\\
&=\frac{a}{2t^2_{u_n}}+\frac{b}{4}-\int\limits_\Omega |u_n|^4\frac{F(x,t_{u_n}u_n)}{|t_{u_n}u_n|^4}dx.\quad\quad\quad\quad (\star)
\end{align*}
Now since $K$ is compact, the sequence $(u_n)$ has a converging subsequence. We can then assume that $u_n\to u$ in $X$.  By the Sobolev embedding theorem $u_n\to u$ in $L^2(\Omega)$, and up to a subsequence $u_n(x)\to u(x)$ a.e. in $\Omega$. Clearly $\|u\|=1$, and consequently $u\neq0$ and $|t_{u_n}u_n|\to\infty$. We point out here that $(f_2)$ and $(f_4)$ imply $F(x,u)\geq0$ for all $(x,u)\in\Omega\times\mathbb{R}$. Hence by using Fatou's lemma and $(f_3)$ we obtain, by passing to the limit $n\to\infty$ in $(\star)$, the contradiction $0\leq-\infty$. Consequently, there exists $C_K>0$ such that $t_u\leq C_K$ for every $u\in K$.
\end{enumerate}
\end{proof}
Now we consider the following mappings.
\begin{equation*}
    M:S\to\mathcal{N},\quad \quad M(u):=t_uu,
\end{equation*}
\begin{equation*}
    \Psi:S\to\mathbb{R},\quad \quad \Psi(u):=\Phi\circ M(u).
\end{equation*}
 The next two lemmas are due to Szulkin and Weth \cite{S-W1}. Indeed, Lemmas \ref{tu} and \ref{tusuite} above show that the assumptions in \cite{S-W1} are satisfied.
\begin{lem}[\cite{S-W1}, Proposition $8$]
$M$ defined above is a homeomorphism between $S$ and $\mathcal{N}$ whose inverse $M^{-1}$ is given by
\begin{equation*}
    M^{-1}(u)=\frac{u}{\|u\|},\quad \forall u\in\mathcal{N}.
\end{equation*}
\end{lem}
We recall that a sequence $(u_n)\subset X$ is said to be a Palais-Smale sequence for a functional $\varphi\in\mathcal{C}^1(X,\mathbb{R})$ if
\begin{equation*}
    \varphi'(u_n)\to0 \quad \text{and}\quad \sup_n|\varphi(u_n)|<\infty.
\end{equation*}
If every such sequence has a convergent subsequence, then $\varphi$ is said to satisfy the Palais-Smale condition.
\begin{lem}[\cite{S-W1}, Corollary $10$]\label{reduction}
\begin{itemize}
  \item [(a)] $\Psi\in\mathcal{C}^1(S,\mathbb{R})$ and
  \begin{equation*}
    \big<\Psi'(u),v\big>=\|M(u)\|\big<\Phi'(M(u)),v\big> \textnormal{ for all }v\in T_u(S),
  \end{equation*}
  where $T_u(S)$ is the tangent space of $S$ at $u$.
  \item [(b)] If $(u_n)$ is a Palais-Smale sequence for $\Psi,$ then $(M(u_n))$ is a Palais-Smale sequence for $\Phi$. If $(u_n)\subset\mathcal{N}$ is a bounded Palais-Smale sequence for $\Phi$, then $(M^{-1}(u_n))$ is a Palais-Smale sequence for $\Psi$.
  \item [(c)] $u$ is a critical point of $\Psi$ if and only if $M(u)$ is a nontrivial critical point of $\Phi$. Moreover, the corresponding critical values coincide and $inf_{S}\Psi=inf_\mathcal{N}\Phi$.
  \item [(d)] If $\Phi$ is even, then so is $\Psi$.
\end{itemize}
\end{lem}
Finally our multiplicity result will be deduced from the following lemma.
\begin{lem}[\cite{S}]\label{szulkin}
Let $X$ be an infinite dimensional Hilbert space and let $J\in\mathcal{C}^1(S,\mathbb{R})$ be even. If $J$ is bounded below and satisfies the Palais-Smale condition, then it possesses infinitely many distinct pairs of critical points.
\end{lem}
%%%%%%%%%%%%%%%%%%%%%%%%%%%%%%%%%%%%%%%%%%%%%%%%%%%%
%%%%%%%%%%%%%%%%%%%%%%%%%%%%%%%%%%%%%%%%%%%%%%%%%%%%
\section{Proof of the main result}\label{section2}
%%%%%%%%%%%%%%%%%%%%%%%%%%%%%%%%%%%%%%%%%%%%%%%%%%%%
%%%%%%%%%%%%%%%%%%%%%%%%%%%%%%%%%%%%%%%%%%%%%%%%%%%%
We shall prove our main result by applying Lemma \ref{reduction}. First we verify the Palais-Smale condition.
\begin{lem}\label{ps}
$\Phi|_{\mathcal{N}}$ satisfies the Palais-Smale condition, that is, every Palais-Smale sequence for $\Phi|_{\mathcal{N}}$ has a convergent subsequence.
\end{lem}
\begin{proof}
Let $(u_n)\subset \mathcal{N}$ such that $d:=\sup_{n}\Phi(u_n)<\infty$ and $\Phi'(u_n)\to0$. We want to show that the sequence $(u_n)$ has a convergent subsequence.
\par First we show that $(u_n)$ is bounded. Arguing by contradiction, we assume that $(u_n)$ is unbounded. Hence, up to a subsequence we have $\|u_n\|\to\infty$ and $v_n:=\frac{u_n}{\|u_n\|}\rightharpoonup v$. By definition of $t_{v_n}$ we have for all $t>0$
\begin{equation*}
    \Phi(t_{v_n}v_n)\geq\Phi(tv_n)=\frac{a}{2}t^2+\frac{b}{4}t^4-\int\limits_\Omega F(x,tv_n)dx.
\end{equation*}
Now since $v_n=M^{-1}(u_n)$, then $u_n=t_{v_n}v_n$ and
\begin{equation}\label{majo}
    d\geq\Phi(u_n)\geq\frac{a}{2}t^2+\frac{b}{4}t^4-\int\limits_\Omega F(x,tv_n)dx.
\end{equation}
If $v=0$, then the Rellich-Kondrashov theorem implies that $v_n\to0$ in $L^2(\Omega)$ and in $L^p(\Omega)$. By using \eqref{ine} we deduce that for every $\varepsilon>0$
\begin{equation*}
    \int\limits_\Omega F(x,tv_n)dx\leq\varepsilon t^2|v_n|_2^2+c_\varepsilon t^p|v_n|_p^p\to0.
\end{equation*}
We then obtain by taking the limit $n\to\infty$ in \eqref{majo}
\begin{equation*}
    d\geq\frac{a}{2}t^2+\frac{b}{4}t^4.
\end{equation*}
But this leads to a contradiction if we take $t$ sufficiently large. Consequently we have $v\neq0$.
By \eqref{phi} and the definition of $v_n$ we have
\begin{equation*}
    0\leq\frac{\Phi(u_n)}{\|u_n\|^4}=\frac{a}{2\|u_n\|^2}+\frac{b}{4}-\int\limits_\Omega|v_n|^4\frac{F(x,\|u_n\|v_n)}{\big|\|u_n\|v_n\big|^4}dx.
\end{equation*}
Since $\big|\|u_n\|v_n\big|\to\infty,$ we obtain by using one more time Fatou's lemma the contradiction $0\leq-\infty$.
The sequence $(u_n)$ is then bounded.
\par Up to a subsequence we have $u_n\rightharpoonup u$ in $X$. By Rellich-Kondrashov theorem $u_n\to u$ in $L^p(\Omega)$. One can easily verify, using \eqref{phi} and \eqref{phiprime} that
\begin{align*}
     \big(a+b&\|u_n\|^2\big)\|u_n-u\|^2=\big<\Phi'(u_n)-\Phi'(u),u_n-u\big>\\
&-b\big(\|u_n\|^2-\|u\|^2\big)\int\limits_\Omega\nabla u\nabla(u_n-u)dx
+\int\limits_\Omega(u_n-u)\big(f(x,u_n)-f(x,u)\big)dx.
\end{align*}
Clearly we have
\begin{equation*}
    \big<\Phi'(u_n)-\Phi'(u),u_n-u\big>\to0 \text{ and }\big(\|u_n\|^2-\|u\|^2\big)\int\limits_\Omega\nabla u\nabla(u_n-u)dx\to0.
\end{equation*}
By the H\"{o}lder inequality
\begin{equation*}
    \Big|\int\limits_\Omega(u_n-u)\big(f(x,u_n)-f(x,u)\big)dx\Big|\leq |u_n-u|_p\big|f(x,u_n)-f(x,u)\big|_{\frac{p}{p-1}}.
\end{equation*}
By $(f_1)$ $f$ satisfies the assumption of Theorem $A.2$ in \cite{W}. Hence $\big|f(x,u_n)-f(x,u)\big|_{\frac{p}{p-1}}\to0$ and consequently
\begin{equation*}
     \big(a+b\|u_n\|^2\big)\|u_n-u\|^2\to0,
\end{equation*}
which implies that $u_n\to u$ in $X$.
\end{proof}
\begin{proof}[Proof of Theorem \ref{main}]
We know from Lemma \ref{reduction}-$(a)$ that $\Psi$ is of class $\mathcal{C}^1$ on $S$. Since $\Psi$ is also bounded below on $S$, Ekeland's variational principle yields the existence of a sequence $(u_n)\subset S$ such that
\begin{equation*}
    \Psi(u_n)\to\inf_S\Psi\quad\text{ and }\quad \Psi'(u_n)\to0.
\end{equation*}
By Lemma \ref{reduction} the sequence $\big(v_n:=M(u_n)\big)\subset\mathcal{N}$ is a Palais-Smale sequence for $\Phi$. By Lemma \ref{ps}, we have $v_n\to v$ up to a subsequence. Since $M$ is a homeomorphism we deduce that $u_n\to u:=M^{-1}(v)$. Hence $\Psi'(u)=0$ and $\Psi(u)=\inf_S \Psi$. By Lemma \ref{reduction}-$(c)$ $v$ is a nontrivial critical point of $\Phi$, and
\begin{equation*}
    \Phi(v)=\Psi(u)=\inf_S\Psi=\inf_\mathcal{N}\Phi.
\end{equation*}
It follows that $v$ is a ground state solution of \eqref{s}.
\par Now $(f_5)$ implies that $\Phi$ is even. Hence by Lemma \ref{reduction}-$(d)$ $\Psi$ is also even. We have seen above that $\Psi\in\mathcal{C}^1(S,\mathbb{R})$ is bounded below and satisfies the Palais-Smale condition. It then follows from Lemma \ref{szulkin} that $\Psi$ has infinite many distinct pairs of critical points. Hence $\Phi$ has infinitely many critical points by Lemma \ref{reduction}, and consequently \eqref{s} has infinitely many solutions.
\end{proof}
Finally, we present an example to illustrate that there is a nonlinear function $f$ which satisfies the
conditions $(f_1)-(f_5)$, but does not satisfy the condition \eqref{AR}.
\begin{exa}
\textnormal{Let $f(x,u)=u^3\ln\big(1+|u|\big)$. Integrating by parts we get
\begin{equation*}
F(x,u)=\frac{1}{4}u^4\ln\big(1+|u|\big)-\frac{1}{4}\Big(\frac{1}{4}u^4-\frac{1}{3}|u|^3+\frac{1}{2}|u|^2-|u|-\ln\big(1+|u|\big)\Big).
\end{equation*}
It is readily seen that the assumptions $(f_1)-(f_5)$ are satisfied.\\
Now it is well known that integrating \eqref{AR} yields the existence of a constant $c_1>0$ such that $F(x,u)\geq c_1|u|^\mu$ for $|u|$ large. Therefore if \eqref{AR} is satisfied for our example above, then we have for $|u|$ large
\begin{equation*}
\frac{1}{4}u^4\ln\big(1+|u|\big)-\frac{1}{4}\Big(\frac{1}{4}u^4-\frac{1}{3}|u|^3+\frac{1}{2}|u|^2-|u|-\ln\big(1+|u|\big)\Big)\geq c_1|u|^\mu.
\end{equation*}
Dividing the two members of this inequality by $|u|^\mu$ and letting $|u|\to\infty$ we get, since $\mu>4$, the contradiction $0\geq c_1$. This shows that the condition \eqref{AR} is not satisfied in our case.}
\end{exa}
\section*{Acknowledgement}
We are grateful to the referees for their careful reading of the paper, and for a number of helpful comments for improvement in this article.
%%%%%%%%%%%%%%%%%%%%%%%%%%%%%%%%%%%%%%%%%%%%%%%%%%%%
%%%%%%%%%%%%%%%%%%%%%%%%%%%%%%%%%%%%%%%%%%%%%%%%%%%%
%
%\section*{References}

%
%
\end{document}